\documentclass[12pt,final]{amsart}
 
\usepackage{tikz} 
\usepackage{a4wide}
\usepackage[notcite,notref,color]{showkeys} 

\usepackage{amsmath,amssymb,amsthm,bbm} 
\usepackage{enumitem} 
\usepackage[overload]{textcase}
 
\usepackage{epsfig,psfrag,color,float,graphics} 
\usepackage{verbatim}

\usepackage[colorlinks,bookmarks,linkcolor=black,citecolor=black]{hyperref} 
\usepackage{bookmark}
\usepackage{amsthm}

\theoremstyle{plain}

\setlist{itemsep=3pt,parsep=0pt,topsep=2pt,partopsep=0pt}  
\setlist{leftmargin=2.5\parindent} 

\def\endofClaim{\hfill\scalebox{.6}{$\Box$}}

\let\subset\subseteq  
\let\eps\varepsilon 
\let\rho\varrho 

\def\cR{\mathcal{R}}

\def\leq{\leqslant}
\def\geq{\geqslant}

 \newtheorem*{theorem*}{Theorem}

\newtheorem{theorem}{Theorem}

\newtheorem{prop}[theorem] {Proposition}

\newtheorem{conj}[theorem]{Conjecture}

\newtheorem{claim}[theorem]{Claim} 
\theoremstyle{definition}

\theoremstyle{remark}

\newcommand{\oldqed}{}

\newenvironment{claimproof}[1][Proof]{
  \renewcommand{\oldqed}{\qedsymbol}
  \renewcommand{\qedsymbol}{\endofClaim}
  \begin{proof}[#1]
}{
  \end{proof}
  \renewcommand{\qedsymbol}{\oldqed}
}

\def\NN{\mathbb{N}}

\usetikzlibrary{calc, shapes, backgrounds}
\usetikzlibrary{arrows,decorations.pathmorphing,positioning,fit,petri}

\title{Ramsey Numbers of Connected Clique Matchings}
\author{Barnaby Roberts}
\address{Department of Mathematics, London School of Economics, Houghton Street, London
WC2A 2AE, U.K.}
\email{b.j.roberts@lse.ac.uk}
\date{\today}
\begin{document}

\begin{abstract}
We determine the Ramsey number of a connected clique matching. That is, we show that if $G$ is a $2$-edge-coloured 
complete graph on $(r^2-r-1)n-r+1$ vertices, then there is a 
monochromatic connected subgraph containing $n$ disjoint copies of 
$K_r$, and that this number of vertices cannot be reduced.
\end{abstract}
\maketitle
\section{Introduction}
For a graph $G$, the Ramsey number $R(G)$ is defined to be the smallest integer $N$ such that every 2-colouring of the edges of the complete graph on $N$ vertices contains a monochromatic subgraph isomorphic to $G$.
The most fundamental problem in Ramsey theory is determining the order of magnitude of the Ramsey numbers of cliques.
An exponential upper bound was given by Erd\H{o}s and Szekeres \cite{ErdSzek}, and in an early use of the probabilistic method Erd\H{o}s \cite{ErdLB} gave an exponential lower bound.
In spite of progress being made on both upper and lower bounds (see \cite{ConUB,SpenLB}) even the size of the exponent is not known asymptotically.
It becomes easier however to look for multiple copies of the cliques.

The Ramsey numbers of multiple copies of graphs were studied in \cite{BurrErdSpen} by Burr, Erd\H{o}s, and Spencer who determined the Ramsey number of $nK_3$ exactly and of multiple copies of a general graph $G$ up to a constant depending only on $G$.
In particular they showed $R(nK_3)=5n$ and for $r\geq 4$ that $(2r-1)n-1 \leq R(nK_r)\leq (2r-1)n +C_r$, determining the Ramsey number of a $K_r$-matching up to a constant.

The aim of this note is to add a connectivity requirement, studying the Ramsey numbers of connected copies of cliques.
Although not technically a Ramsey number by the definition given above, we let $R(c(nH))$ denote the least $N$ such that every $2$-colouring of the edges of $K_N$ contains a monochromatic copy of $nH$ in a connected component of the same colour.
This was first studied by Gy\'arf\'as and S\'ark\"{o}zy \cite{GyaSark} who solved the problem for $H=K_3$ showing $R(c(nK_3))=7n-2$.
We solve the problem for all larger cliques.

\begin{theorem}\label{thm:main}
For $r\geq 4$ and $n \geq R(K_r)$ we have \[R(c(nK_r)) = (r^2 -r +1)n -r +1\,. \]
\end{theorem}

The connectivity requirement here proves to be very significant as for large values of $r$ the Ramsey number of a connected $K_r$-matching is approximately $r/2$ times larger than that of a standard $K_r$-matching.

The lower bound in Theorem~\ref{thm:main} is a special case of the following observation of Burr~\cite{Burr}.
\begin{prop}\label{prop:Burr}
For every connected graph $G$ containing at least one edge
\begin{equation*}
R(G) \geq (\chi(G)-1)(|G|-1) + \sigma(G)
\end{equation*}
where $\sigma(G)$ is the smallest size of a colour class over all $\chi(G)$-proper-colourings of $G$.
\end{prop}
To see this partition the vertices of $K_N$, with $N=(\chi(G)-1)(|G|-1) + \sigma(G)-1$, into parts $X_1,\ldots,X_{\chi(G)-1}$ and $Y$ such that $|Y|=\sigma(G)-1$ and $|X_i|=|G|-1$ for each $i$.
Colour all edges within each part with blue and colour all edges between different parts with red.
There cannot be a blue copy of $G$ since all connected components of blue are too small.
There cannot be a copy of $G$ in red since the red edges form a $\chi(G)$-partite graph where the smallest part is too small.

This construction led Burr to conjecture that for any $\Delta,k \in \NN$ there exists an $n_0$ such that any connected graph $G$ on $n \geq n_0$ vertices with chromatic number $k$ and maximum degree at most $\Delta$ satisfies $R(G)=(k-1)(n-1) +\sigma(G)$.
Burr's conjecture was proven to be false by Graham, R\"odl, and Rucinski~\cite{GraRodRuc}.
Our result shows that Burr's conjecture does hold in the case of connected copies of cliques.
We conjecture the following
\begin{conj}
For any graph $H$ there exists $n_0(H)$ such that for all $n \geq n_0$ we have
\[R(c(nH))=(\chi(H)-1)(n|H|-1)+n\sigma(H)\,.  \]
\end{conj}
The lower bound holds by Proposition~\ref{prop:Burr}.

In \cite{GyaSark}  Gy\'arf\'as and S\'ark\"{o}zy used their method for determining $R(c(nK_3))$ along with the Regularity Lemma to estimate asymptotically the Ramsey number of a $C_n^2$ that is missing a constant number of edges.  We suspect, but have not checked all technicalities, that we could obtain a similar result, showing that for any $r$ and any $\eps>0$ there exists $c_r$ such that any $2$-colouring of $K_N$ with $N\geq (1+\eps)(r^2 - r+1)n$ contains a monochromatic copy of a $C_{rn}^{r-1}$ missing some $c_r$ edges.

There is another stronger connectivity requirement on disjoint cliques for which it is interesting to study the Ramsey numbers.
Suppose $P$ and $Q$ are copies of $K_r$.  We say they are $K_r$-connected if there exist copies $C_1,\ldots,C_t$ of $K_r$ such that in the sequence $P,C_1,\ldots,C_t,Q,$ consecutive copies of $K_r$ share $r-1$ vertices.
A construction of Allen, Brightwell, and Skokan in \cite{AllBriSko} shows that if we wish to find a monochromatic set of $n$ disjoint copies of $K_r$ which are $K_r$-connected in $K_N$ we need $N \geq nr^2 -2r +2$.
This construction was given as a lower bound for the Ramsey number $R(C_{rn}^{r-1})$ where $C_m^k$ is the $m$-th power of a cycle, obtained by joining all vertices of $C_m$ at distance at most $k$.
In \cite{AllSkoRob}, with Allen and Skokan we determine the Ramsey number $R(C^2_m)$ for large $m$ via studying the Ramsey numbers of $K_3$-connected copies of disjoint triangles.

\section{Proof of Theorem~\ref{thm:main}}
The lower bound for Theorem~\ref{thm:main} follows from the construction of Burr given in the previous section.
Before giving a proof of the upper bound we sketch the main ideas.

For $r\geq 4$ and $n\geq R(K_r)$, consider a $2$-colouring of $G=K_N$ with $N= (r^2 -r +1)n -r +1$.
Since in any $2$-colouring of a complete graph one of the colours is connected, we assume $G$ is connected in red and look for either a red copy of $nK_r$ or a blue connected copy of $nK_r$.
We then show that $G$ can be partitioned into a maximum set of disjoint red copies of $K_r$ and a set of $r-1$ large blue cliques such that between any two of the blue cliques all edges are red.
We call these blue cliques $B_1,\ldots,B_{r-1}$, and the union of their vertex sets $B$.
We let $\cR$ denote the maximum set of red copies of $K_r$, and we let $R$ denote the vertex set of $\cR$.
We then consider edges between the red copies of $K_r$ in $\cR$ and the sets $B_1,\ldots,B_{r-1}$.
We show that each clique of $\cR$ is of one of two types (see claim~\ref{cl:cliquetypes}) with regards to how the edges between that clique and $B$ are coloured.
Furthermore each type gives a way of assigning vertices of the red $K_r$ to some $B_i$ such that almost all of the edges between the assigned vertex and $B_i$ are blue.
For each $i=1,\ldots,r-1$, we let $D_i$ denote the union of $B_i$ along with the vertices of $R$ that were assigned to $B_i$.
We then use an averaging argument to show that there exists an $i$ such that $|D_i|\geq rn$ and we look for a blue $nK_r$ on this $D_i$.
Since $B_i$ was a blue clique and vertices assigned to $B_i$ were connected to $B_i$ in blue, this $nK_r$ is connected in blue.
If more than $(r-1)n$ of the vertices of $D_i$ came from $B_i$ we can find the $nK_r$ greedily.
Otherwise we use more information coming from the types of red $K_r$ vertices assigned to $B_i$ came from to find an $nK_r$.
There is one special case for which this method fails.
If this occurs we find a connected $nK_r$ in a different $D_{i'}$.

\begin{proof}[Proof of Theorem~\ref{thm:main}]
For $r\geq 4$ and $n\geq R(K_r)$, let $N= (r^2 -r +1)n -r +1$ and consider a red/blue edge colouring of $G=K_N$.
Any $2$-colouring of $K_N$ is connected in one of the colour, since if Blue has more than one connected component, Red contains a complete multi-partite graph between blue components which connects the graph in red.
Without loss of generality we assume $G$ is connected in red.
Consider a maximal set of vertex disjoint red copies of $K_r$.
We call this set of cliques $\cR$ and its vertex set $R$.
Note that $|R| \leq r(n-1)$, otherwise we would be done.
Greedily looking for vertex disjoint blue copies of $K_r$ on the rest of the graph results in covering all but at most $R(K_r)-1$ vertices of the graph with monochromatic copies of $K_r$.
Call this set of uncovered vertices $Z$.

Let $B$ denote the vertices covered by these blue cliques.
We next partition $B$ into blue connected components.

Between components all edges are red, so there cannot be more than $r-1$ such components or we would have a new red $K_r$ to add to $\cR$.
This would contradict the maximality of $\cR$.
There also cannot be fewer than $r-1$ such components, or, since each component is no larger than $r(n-1)$, we would have 
\[|B|+|R|+|Z| \leq (r-2)r(n-1)+ r(n-1) +R(K_r)-1 = (r^2 -r)n -r^2 +r -1 +R(K_r),\]
which is a contradiction, since for $n>R(K_r)-r^2 +2r -2$ this is less than $N$.
Therefore there must be exactly $r-1$ blue components and we call them $B'_1,\ldots ,B'_{r-1}$.
If any of these components contained a red edge we would have another red $K_r$ to add to $\cR$, and so each $B'_i$ is a blue clique.
Since the blue components are disconnected, each vertex in $Z$ can be adjacent in blue to vertices of at most one $B'_i$.
If any vertex in $Z$ had a red neighbour in each $B'_i$ we would have a new red $K_r$ to add to $\cR$.
Therefore all vertices in $Z$ are adjacent to $r-2$ of the $B'_i$ entirely in red and the remaining blue component entirely in blue.
For each $i=1,\ldots,r-1$, form $B_i$ by adding to $B'_i$ the vertices of $Z$ that are adjacent in blue to all of $B'_i$.
Note that there cannot be any red edges in any $B_i$ or we would have a new red $K_r$ to add to $\cR$ and also between any two distinct $B_i$ all edges are red since they are not blue connected.

We now consider the colour of edges between the cliques of $\cR$ and the sets $B_1,\ldots,B_{r-1}$.
First, recalling that all vertices in $R$ are adjacent in blue to vertices in at most one $B_i$, we say a vertex of $R$ is \emph{paired with $B_i$} if it is adjacent in red to all $B_j$ for all $j \neq i$.
The following claim identifies two possible properties of cliques of $\cR$.

\begin{claim}\label{cl:cliquetypes}
Let $C$ be a red $K_r$ from $\cR$.  Then one of the following holds:
\begin{enumerate}
\item[(i)]\label{cl:it:1} For each $B_i$ there is a vertex in $C$ adjacent in blue to all but at most one vertex of $B_i$.
\item[(ii)]\label{cl:it:2} For all but two values of $i$ there is exactly one vertex of $C$ which is adjacent in blue to all vertices of $B_i$.  Furthermore there is a $j$ such that the three remaining vertices of $C$ are adjacent to all of $B_j$.
\end{enumerate}

\end{claim}
\begin{claimproof}
Consider some $C \in \cR$.
Each vertex in $C$ can have blue neighbours in at most one $B_i$ and so is entirely adjacent in red to the remaining $r-2$ blue components.

Suppose firstly that some $c_i \in C$ is paired with $B_i$ but also has at least one neighbour in red in $B_i$.
Then if two other vertices $c_j,c'_j \in C$ were paired with the same $B_j$, we could break up $C$ to make two new red copies of $K_r$, contradicting the maximality of $\cR$.
One of these red cliques uses $c_i$, its red neighbour in $B_i$ and vertices of $B_k$ for $k \neq i$.
The other uses $c_j$, $c'_j$ and vertices of each $B_k$ for $k \neq j$.
Therefore, if such a $c_i \in C$ exists, all remaining vertices of $C$ are paired with distinct components $B_j$.
Furthermore, if any of them has more than one red neighbour in the set they are paired with, we could make two new red copies of $K_r$, again contradicting the maximality of $\cR$.
One of these would use $c_i$ as before, and the other would use the other vertex with red neighbours in the set it is paired with.
Thus, if $c_i$ as above exists, $(i)$ holds.

Secondly, suppose all vertices in $C$ are adjacent entirely in blue to the set they are paired with.
There cannot be two vertices $c_i, c'_i$ paired with some $B_i$ and another two vertices $c_j, c'_j$ paired with some $B_j$.
If there were, we would create two new red copies of $K_r$ using $c_i$, $c'_i$ and a vertex from each $B_k$ with $k \neq i$ for one, and $c_j$, $c'_j$ and a vertex from each $B_k$ with $k \neq j$ for the other.
This would again break the maximality of $\cR$.
There also cannot be four or more vertices paired with the same $B_i$, or we could use two pairs of them to make two new red copies of $K_r$ along with vertices from $B_k$ for $k \neq i$.
Therefore we either have some $B_i$ with two vertices of $C$ paired with it and every other $B_j$ has one vertex paired with them, as in $(i)$, or we have some $B_i$ with three vertices paired with it and all but one of the remaining components $B_j$ have one vertex paired with them, as in $(ii)$.
\end{claimproof}

The next claim tells us that if we have two sets of three vertices from cliques of $\cR$ that are paired with the same $B_i$, then the edges between these two sets are all blue.
\begin{claim}\label{cl:bluebetween}
Let $C$ and $C'$ be cliques of $\cR$ with vertices $x_1,x_2,x_3 \in C$ and $y_1,y_2,y_3 \in C'$ all paired with the same $B_i$.
Then all the edges $x_jy_k$ for $j,k \in \{1,2,3\}$ are blue.
\end{claim}
\begin{claimproof}
Suppose for contradiction and without loss of generality that $x_1y_1$ is red.
Then we could create three new red copies of $K_r$ at the cost of $C$ and $C'$, contradicting the maximality of $\cR$.
These three copies of $K_r$ would use the pairs of vertices $\{x_1,y_1\}$, $\{x_2,x_3\}$ and $\{y_2,y_3\}$ along with vertices from $B_j$ for $j \neq i$.
\end{claimproof}

We now use Claim~\ref{cl:cliquetypes} to partition the red cliques of $\cR$ depending on which option of the claim they satisfy.
Let $S \subset \cR$ be the set of cliques satisfying $(i)$, and $T \subset \cR$ be the set of cliques satisfying $(ii)$.
For each $C \in S$ and each $B_i$ there is at least one vertex of $C$ which is adjacent in blue to all but at most one vertex of $B_i$.
For each $i$, construct $S_i$ by selecting one such vertex from each $C \in S$.
For each $C \in T$, all vertices are adjacent entirely in blue to exactly one $B_i$.
For each $i$, construct $T_i$ by taking the vertices of each clique of $T$ that are entirely adjacent in blue to $B_i$.
Let $D_i = B_i \cup S_i \cup T_i$.
We further split up $T_i$, into sets $T_i^*$ and $T_i^\Delta$, depending on whether one or three vertices from that red $K_r$ were added to $T_i$.
Given a vertex $u \in T_i$ let $C$ be the clique of $T$ containing $u$.
If $u$ is the only vertex of $C$ belonging to $T_i$ then $u$ belongs to $T_i^*$.
Otherwise three vertices of $C$ must belong to $T_i$ in which case all three belong to $T_i^\Delta$.
Observe that $|T_i|=|T_i^*|+|T_i^\Delta|$ and $n-|S_i|-|T_i^*|\geq \frac13 |T_i^\Delta|+1$.

We shall find a blue connected copy of $nK_r$ on a $D_i$ such that either $|D_i|\geq rn+1$, or $|D_i|\geq rn$ and $|T_i^\Delta| \neq 3$.
We first proceed to show that such a $D_i$ exists.

Because the $D_i$ for $i=1,\ldots,r-1$ cover all vertices of the graph except for one from each clique of $S$, we have
\begin{equation*}
N=(r^2 -r +1)n-r+1=|S|+ \sum_{i=1}^{r-1} |D_i| \,.
\end{equation*}

By averaging and using $|S|\leq n-1$, there exists an $i$ such that
\begin{equation*}
|D_i|\geq \bigg(r + \frac{1}{r-1}\bigg)n -1 -\frac{|S|}{r-1}\geq rn -1 +\frac{1}{r-1}\,.
\end{equation*}
Since $|D_i|$ is an integer it is at least $rn$.
Furthermore if $|S|\leq n-r$ it is at least $rn+1$.
Therefore if $|S|\leq n-r$ we can find a suitable $D_i$.
If $|S|$ is larger than this, the next claim shows that either we still have a $D_i$ with $|D_i|\geq rn+1$ or we have at least two of size at least $rn$.
In the latter case, we can find one of these with $|T_i^\Delta|\neq 3$.
\begin{claim}\label{cl:advave}
Suppose $|S|=n-\ell$ for some $1\leq \ell \leq r-1$.  Then either there exist $\ell$ values of $i$ for which
$|D_i| \geq rn$
or there exists one value of $i$ for which $|D_i| \geq rn +1$.
\end{claim}
\begin{claimproof}
Suppose for contradiction that the $\ell -1$ largest sets $D_i$ all have at most $rn$ vertices whilst all others have at most $rn-1$.
This gives
\[|S|+\sum_{i=1}^{r-1}|D_i| \leq n-\ell +(\ell-1)rn +(r-\ell)(rn-1)=(r^2-r+1)n-r <N\]
achieving a contradiction.
\end{claimproof}
Since $|T|\leq n-1 -|S|$, we see $|T| \leq \ell -1 $, and so if $|S|\geq n-r+1$ we can find a value of $i$ such that either $|D_i|$ is at least $rn+1$ or $|D_i|\geq rn$ and $|T_i^\Delta| \neq 3$.

We begin using only the assumption $|D_i|\geq rn$. 
Then \[ |B_i|\geq rn-|S_i|-|T_i^*|-|T_i^\Delta|\]
and so $|B_i|-(r-1)(|S_i|+|T_i^*|) \geq r(n-|S_i|-|T_i^*|)-|T_i^\Delta|\geq (\frac{r}{3}-1)|T_i^\Delta| +r \geq 1$.
Since all vertices of $S_i \cup T_i^*$ are adjacent in blue to all but at most one vertex of $B_i$, we can extend all vertices of $S_i \cup T_i^*$ to disjoint blue copies of $K_r$ using $B_i$.
We now look to find $n-|S_i|-|T_i^*|$ disjoint blue copies $K_r$ on the remaining vertices of $B_i$ and $T_i^\Delta$.
Let $\tilde{B_i}$ denote the remaining vertices of $B_i$, noting that $|\tilde{B_i}| \geq (\frac{r}{3}-1)|T_i^\Delta|+r$.
Recall that edges from $T_i^\Delta$ to $B_i$ are blue and between different red triangles of $T_i^\Delta$ edges are blue.

If $|T_i^\Delta|=0$, then since $B_i$ is a blue clique we can find the remaining blue copies of $K_r$ entirely on the rest of $B_i$.

Remembering that $|T_i^\Delta|$ is a multiple of three, suppose $|T_i^\Delta|\geq 6$.  If $|T_i^\Delta|\geq 3r$, we first take blue copies of $K_r$ on $T_i^\Delta$ such that the vertices of $T_i^\Delta$ that are not covered by these blue cliques consist of $t$ red triangles with $2\leq t \leq r-1$.
We then cover the rest of $T_i^\Delta$ with blue copies of $K_r$ by greedily taking one vertex from each of the $t$ red triangles and extending this set of $t$ vertices to a blue $K_r$ using vertices from $\tilde{B_i}$.
In this process three blue copies of $K_r$ use both vertices of $T_i^\Delta$ and $\tilde{B_i}$, one for each vertex of the red triangles, and each such blue $K_r$ used $r-t$ vertices from $\tilde{B_i}$.
So long as $|\tilde{B_i}|\geq 3(r-t)$, this greedy procedure is successful.
Since we have $|\tilde{B_i}| \geq (\tfrac{r}{3}-1)|T_i^\Delta|+r$, we also have $|\tilde{B_i}|\geq 3t(\tfrac{r}{3}-1)+r$, which is at least $3(r-t)$ for $t \geq 2$.

Finally, suppose $|T_i^\Delta|=3$. Using Claim~\ref{cl:advave} we may assume $|D_i|\geq rn+1$. 
We have that $|\tilde{B_i}|\geq rn+1 -(r-1)(|S_i|+|T_i^*|) \geq \big(\tfrac{r}{3}-1\big)|T_i^\Delta| +r+1=2(r-1)$, and so we can extend two of the vertices of $T_i^\Delta$ to blue copies of $K_r$ using $\tilde{B_i}$. If necessary we then cover the rest of $\tilde{B_i}$ with more copies of $K_r$.
This completes the proof.
\end{proof}
\bibliographystyle{plain}
\bibliography{Bib3}
\end{document}